\documentclass[a4paper,11pt]{amsart}
\usepackage{cite}
\usepackage{fullpage}
\usepackage{qtree}
\usepackage{amssymb}
\usepackage{MnSymbol}

\usepackage{amsthm}
\usepackage{amsmath} 
\usepackage{mathtools}
\usepackage{graphicx}
\usepackage[all,2cell]{xy}
\SelectTips{cm}{}
\usepackage[mathscr]{euscript}
\usepackage{listings}
\usepackage{young}
\usepackage{ytableau}
\usepackage{enumitem}
\usepackage{tikz}
\usetikzlibrary{arrows}
\usepackage{subcaption}
\usepackage{youngtab}
\usepackage{tikz-cd}
\usepackage{ifthen}
\usepackage{color}
\usepackage[normalem]{ulem}
\usetikzlibrary{arrows}

\newcommand{\Pj}{\mathbb{P}}
\newcommand{\Z}{\mathbb{Z}}

\newcommand{\Gr}{\mathrm{Gr}}

\newcommand{\eat}[1]{}

\newcommand{\C}{\mathbb{C}}

\usepackage[citecolor=black,urlcolor=blue,bookmarks=false,hypertexnames=true]{hyperref} 
\newtheorem{theorem}{Theorem}[section]

\newtheorem{corollary}[theorem]{Corollary}

\newtheorem{lemma}[theorem]{Lemma}

\newtheorem{remark}[theorem]{Remark}
\newtheorem{question}[theorem]{Question}

\makeatletter
\DeclareRobustCommand*\cal{\@fontswitch\relax\mathcal}
\makeatother

\date{}
\begin{document}

\title[Morphisms from projective spaces to flags of minimal parabolics]{Morphisms from projective spaces to flags of minimal parabolic subgroups}
\author{Sarjick Bakshi}
\thanks{Indian Institute of Technology, Kanpur, India, {\tt sarjick91@gmail.com}, Corresponding author}
\author{A J Parameswaran}
\thanks{Kerala School of Mathematics, Kerala, India, {\tt param@ksom.res.in}}



\maketitle

\begin{abstract}
We classify nonconstant morphisms $\Pj^m \to G/P$ for $m \le 4$ when $G = SL(n,\C)$ (type~$A$) for a minimal parabolic subgroup $P$. 
Using the Borel presentation of cohomology and explicit Schubert intersection identities, 
we show that there is no nonconstant morphism $\Pj^2 \to G/B$; for minimal parabolic subgroup $P_{\alpha_i}$, 
there are no nonconstant morphisms $\Pj^3 \to G/P_{\alpha_i}$ when $i \in \{1, n-1\}$, 
while such morphisms exist for $1 < i < n-1$; and, after correcting an earlier error (pointed out by Yanjie Li), 
we give an elementary proof that there is no nonconstant morphism $\Pj^4 \to G/P_{\alpha_i}$ for any minimal parabolic subgroup. 
The proofs are elementary and cohomological.
\end{abstract}

\section{Introduction}

We study regular morphisms $\phi : \Pj^m \to G/P$, where $G = SL(n,\C)$ and $P \supset B$ is a parabolic subgroup.
Our approach is entirely cohomological.
We make systematic use of the Borel presentation~\cite{Borelcohoring}
\[
  H^\ast(G/B,\Z) \;\cong\; \Z[x_1,\dots,x_n]/\mathcal I,
\]
where $\mathcal I$ is the ideal generated by symmetric polynomials of positive degree, together with the description
\[
  H^\ast(G/P) \;\cong\; H^\ast(G/B)^{W_P}
\]
due to Reiner--Woo--Yong~\cite{Reineretal}.
Pulling Schubert intersection relations back along $\phi$ to
\[
  H^\ast(\Pj^m) \;\cong\; \Z[t]/(t^{m+1})
\]
produces explicit arithmetic constraints which, in low dimensions, force the vanishing of all positive-degree cohomology classes.
This mechanism underlies all rigidity and nonexistence results proved in this paper.

Let $k \le n$ be positive integers and let $Gr(k,n)$ denote the Grassmannian of $k$-dimensional subspaces of an $n$-dimensional complex vector space.
The Grassmannian is itself a partial flag variety $G/P$ corresponding to a maximal parabolic subgroup.
Morphisms from projective spaces to Grassmannians were studied extensively by Tango in a series of foundational papers~\cite{Tango1,Tango2,Tango3}.
In particular, Tango showed in~\cite{Tango1} that there is no morphism $\Pj^m \to Gr(k,n)$ for $m \ge n$, and in~\cite{Tangobundle} constructed indecomposable globally generated vector bundles of rank $n-1$ on $\Pj^n$.
Later, while addressing questions of Lazarsfeld~\cite{Lazarsfeld}, Paranjape--Srinivas~\cite{srinivas} proved that under the assumption $1 \leq k \leq n-k$, a finite surjective morphism $Gr(k,n) \to Gr(\ell,m)$ exists if and only if $(k,n) = (\ell,m)$.
These works naturally motivated a broader investigation of morphisms between projective homogeneous varieties.

Cohomological methods have played a central role in the study of morphisms
between projective homogeneous varieties. Tango’s original arguments relied
on comparisons of Chow rings. Subsequently, Muñoz--Occhetta--Solá Conde
\cite{Munoz-etal} introduced the notion of \emph{effective good divisibility},
refining an earlier concept of good divisibility due to Pan~\cite{Pan}, and
showed that this invariant already suffices to obtain Tango-type rigidity
results. Using effective good divisibility, Naldi--Occhetta
\cite{Naldi--Occhetta} proved that every morphism
$\Gr(k,n)\to\Gr(\ell,m)$ with $n>m$ is constant, and computed the effective
good divisibility of Grassmannians. More recently, Occhetta--Tondelli
\cite{OcchettaTondelli24} showed that a nonconstant morphism
$\Gr(l,n)\to\Gr(k,n)$ with $l\neq 0,n-1$ forces $l=k$ or $l=n-k-1$, in which
case the morphism is necessarily an isomorphism. Independently and almost
contemporaneously, Hu--Li--Liu~\cite{HuLiLiu23} obtained general nonexistence
results for morphisms between rational homogeneous varieties of arbitrary
Lie type, using a different, Lie-theoretic and combinatorial approach to
effective good divisibility.

In this paper we focus on morphisms from low-dimensional projective spaces to partial flag varieties of type~$A$.
Our first basic result is the following.

\begin{theorem}\label{mapsfromP2}
There is no nonconstant morphism from $\Pj^2$ to $G/B$.
\end{theorem}

\medskip

\noindent
\textbf{Remark.}
After the first version of this work was completed, we became aware that Shrawan Kumar independently proved the nonexistence of morphisms $\Pj^2 \to G/B$ in~\cite{kumarsconj}. In fact he effectively got a general version of our result. We however keep our proof for completion. Shrawan Kumar has formulated a general conjectural framework governing the existence of morphisms between complex flag varieties \cite{kumarsconj}.

\medskip

Let $P \supset B$ be a parabolic subgroup and denote by $\mathrm{rank}(G/P)$ the rank of the Picard group $\mathrm{Pic}(G/P)$.
For example, $\mathrm{rank}(Gr(k,n)) = 1$, $\mathrm{rank}(G/B) = n-1$, and $\mathrm{rank}(G/P) = n-2$ for a minimal parabolic subgroup $P$.
Motivated by Theorem~\ref{mapsfromP2} and by Tango’s results, it is natural to ask:

\begin{question}\label{question}
Classify morphisms $\Pj^m \to G/P$ with $m = n + 1 - \mathrm{rank}(G/P)$?
\end{question}

We investigate this question for minimal parabolic subgroups.
Let $\alpha_1,\dots,\alpha_{n-1}$ denote the simple roots of $SL(n,\C)$, and for each $\alpha_j$ let
\[
  P_{\alpha_j} := B \cup Bs_{\alpha_j}B
\]
be the corresponding minimal parabolic subgroup.
Our main results are summarized as follows (detailed proofs appear in Sections~\ref{sec:P2}--\ref{sec:P4}).

\begin{itemize}[leftmargin=*,itemsep=3pt]
\item[(A)] (\S\ref{sec:P2}) There is no nonconstant morphism $\Pj^2 \to G/B$.
\item[(B)] (\S\ref{sec:P3}) For minimal parabolic subgroups $P_{\alpha_i}$, there is no nonconstant morphism
$\Pj^3 \to G/P_{\alpha_i}$ when $i \in \{1,n-1\}$, while for each $1 < i < n-1$ there exists a nonconstant morphism
$\Pj^3 \to G/P_{\alpha_i}$.
\item[(C)] (\S\ref{sec:P4}) For any minimal parabolic subgroup $P_{\alpha_i}$, there is no nonconstant morphism
$\Pj^4 \to G/P_{\alpha_i}$.
\end{itemize}

Taken together, these results provide a complete classification of morphisms $\Pj^m \to G/P$ for $m \le 4$ in type~$A$.
The proofs are entirely elementary and explicit: we avoid characteristic-class and stability arguments and instead rely only on Schubert intersection identities and relations among symmetric polynomials.

\noindent
\textbf{Remark.}
After the first version of this paper appeared on the arXiv, Yanjie Li kindly pointed out an error in the proof of Theorem~5.1, namely that Formula~(8) in that argument was incorrect.
As a result, Theorem~6.1 of the original preprint was removed.
In the present version we provide a corrected and completely elementary proof of nonexistence of morphism from $\Pj^4$ to $G/P_{\alpha_i}$ for any minimal parabolic $P_{\alpha_i}$ thereby restoring this result.
After the initial posting we were also informed of related work of Fang and Ren~\cite{FangRen2025}, which proves rigidity results of a different nature but does not subsume the explicit low-dimensional classification carried out here.

\section{Preliminaries}
Let $G = SL(n,\C)$ denote the set of all $n\times n$ matrices with determinant $1$. Let $B$ denote the Borel subgroup of upper triangular matrices, and $T$ denote the maximal torus consisting of diagonal matrices inside $G$. Denote $R$ the root system of $(G,T)$. Let $R^+$ denote the subset of $R$ consisting of positive roots. Let $\epsilon_i$ denote the character of $T$ which sends $diag(t_1,t_2,\ldots, t_n)$ to $t_i$. Let $\alpha_i = \epsilon_i - \epsilon_{i+1}$. Then a subset $S= \{\alpha_1, \ldots, \alpha_{n-1} \}$ of $R^+$ gives a set of simple roots. The Weyl group $W$ is the group generated by the simple reflections $s_{\alpha}$, $\alpha \in S$. In our case, $W$ is the symmetric group in $n$ letters $S_n$. The simple reflections $s_{\alpha_i}$ can be thought of as the transposition of $i$-th and $i+1$-th letter. We would use the one-line notation $(w(1),w(2),\ldots,w(n))$ to denote the permutation $w$ in $S_n$. 

Let $J$ be a subset of $S$. Let $W_J$ denote the subgroup of $W$ generated by $s_{\alpha}$, $\alpha \in J$. For every $J$ we associate a parabolic subgroup $P_J$ as follows 
\[ P_J = \bigsqcup_{w \in W_J} BwB.  \]

The set $W^J = W/W_J$ is called the set of {\em minimal length coset representatives}. Alternatively, we have (see, \cite[Section 2.5]{Billeylakshmibai}) 
\[ W^J = \{ w\in W \mid w(\alpha) > 0 \text{ for all } \alpha \in J \}.
\] 

The {\em full flag variety} is by definition the variety $G/B$. The projective homogeneous space $G/P_J$ is called a {\em partial flag variety} and its Bruhat decomposition is given by 
\[ G/P_J = \bigsqcup_{w \in W^J} BwP_J.
\]
Whenever $W_J$ is generated by one element $s_\alpha$ for $\alpha \in S$, we call the associated parabolic subgroup a {\em minimal parabolic subgroup} and we denote it as $P_{\alpha}$. Note that, $P_{\alpha}/B$ is isomorphic to $\mathbb{P}^1$. Whenever $J$ is obtained from $S$ by removing one simple root $\alpha_k$, we call the associated parabolic subgroup a {\em maximal parabolic subgroup} and we denote it by $P_{\hat{\alpha_k}}$. We recall that the {\em Grassmannian variety} $Gr(k,n)$ of $k$ dimensional subspaces of a $n$-dimensional complex vector space is isomorphic to $G/{P_{\hat{\alpha_k}}}$. Let 
\[ I(k,n) = \{(i_1, i_2, \ldots, i_k) |~ 1 \leq i_1 < i_2 < \cdots < i_{k} \leq n\}.
\]
Let $w = (i_1,i_2, \ldots, i_k) \in I(k,n)$. Let $e_1, e_2, \ldots, e_n$ be the standard basis of $\mathbb{C}^n$. Let $M_i$ denote the vector space spanned by $e_1,e_2, \ldots ,e_i$. The {\em Schubert cell} $C(w)$ in the Grassmannian is defined as 
\[ C(w) = \{U \in Gr(k,n)|~\rm{dim}(U \cap M_{i_j}) = j, 1\leq j \leq k \}.
\]
The dimension of such a Schubert cell $C(w)$ is given by $\sum_{j}(i_j-j)$. The {\em Schubert variety} $X(w)$ which is the closure of $C(w)$ in Grassmannian can be seen to be
\[ X(w) = \{U \in Gr(k,n)|~\rm{dim}(U \cap M_{i_j}) \geq j, 1\leq j \leq k \}.
\]

Let $R$ denote the polynomial ring $\mathbb{Z}[x_1,x_2,\ldots,x_n]$ in $n$ variables with degree of $x_i$ being $2$. We recall that $S_n$ acts on the variables as 
\[ \sigma(x_i) = x_{\sigma (i)}.
\]
The action extends to an action of $S_n$ on $R$.  A polynomial $f(x_1,x_2,\ldots,x_n)$ in $R$ is {\em symmetric} if and only if 
\[ f(x_1,x_2,\ldots,x_n) = f(\sigma(x_1),\sigma(x_2),\ldots,\sigma(x_n))
\]
for all $\sigma \in S_n$. 
The {\em power sum symmetric polynomial} $p_k(x_1,x_2\ldots,x_k)$ is defined as 
\[  p_k(x_1,x_2\ldots,x_k) = \sum^n_{i=1} x_i^k .
\]

We recall that the subring of invariants $R^{S_n}$ of $R$ is a graded subring and is generated by symmetric polynomials. Let ${\cal{I}}$ denote the ideal generated by symmetric polynomials in positive degree. The power sum symmetric polynomials $p_k(x_1,x_2,\ldots,x_n)$ for $1 \leq k \leq n$ form a set of generators for ${\cal{I}}$. 

Let $X$ be a projective variety. Let $H^{\bullet}(X) = \bigoplus\limits_{d=1}^{n} H^d(X)$ denote the cohomology ring of the variety with integer coefficients. Let $A^{\bullet}(X) = \bigoplus\limits_{d=1}^{n} A^d(X)$ denote its Chow ring. We recall from \cite[Chapter 19]{fultonintersection} that there exists a cycle map 
\[ cy: A^{\bullet}(X) \longrightarrow H^{\bullet}(X). 
\]
Whenever $X$ is a partial flag variety the map $cy$ is an isomorphism (see, \cite[Example 19.1.11]{fultonintersection}) and the cohomologies in odd degrees vanish. When $X$ is the full flag variety $G/B$ we recall

\begin{theorem}\cite[Ehresmann]{Ehresmann} $H^{2d}(G/B)$ has a basis consisting of classes of Schubert varieties $[X(w_0w)]$ where $l(w) = d$ where $w_0$ is the longest word in $W$. 
\end{theorem}

In \cite{Borelcohoring}, Borel, gave a presentation of the cohomology ring using the polynomial ring $R$ and the ideal $\mathcal{I}$
\begin{theorem}\label{borelisom} \cite[Borel]{Borelcohoring} 
$H^{\bullet}(G/B) \cong R/{\cal{I}}$.
\end{theorem}

The results were extended for $G/P$, where $P$ is a parabolic subgroup of $G$ containing $B$ in \cite{Reineretal}. Let $J \subset S$ such that $P = P_J$. We have $W_J$ the subgroup of Weyl group generated by $J$ as above. Since $W_J$ is subgroup of $W$ it also acts on $H^{\bullet}(G/B)$. Reiner--Woo--Yong show that,

\begin{theorem} \cite{Reineretal}[Reiner--Woo--Yong]
$H^{\bullet}(G/P) \cong H^{\bullet}(G/B)^{W_J}$.
\end{theorem}

We observe that 
\[ H^{\bullet}(G/P) \hookrightarrow  H^{\bullet}(G/B). \]

\begin{remark} \label{cohG/P} Under this inclusion we recall from \cite{Reineretal}, the cohomology classes $[X(w)]$ where $w \in W^J$ lies in $H^{\bullet}(G/B)^{W_J}$ and forms a basis of $H^{\bullet}(G/P)$. More precisely, a basis of $H^{2d}(G/B)^{W_J}$ consists of the Schubert classes $[X(w)]$, where $w \in W^J$ and $X(w)$ is a codimension $d$ Schubert subvariety of $G/P$. This can be thought of as a generalisation of Ehresmann's theorem. 
\end{remark}

\section{Morphism from $\mathbb{P}^2$ to $G/B$} \label{sec:P2}

As in the previous section, we have $G = SL(n,\mathbb{C})$, $B$ denotes the Borel subgroup consisting of the diagonal matrices in $G$. We will begin this section by proving the following:

\begin{theorem}\label{consp2} There exists no nonconstant morphism from $\mathbb{P}^2$ to $G/B$.
\end{theorem}
\begin{proof}
Let $\phi$ be such a morphism and 
\[ \phi^{*i}: H^{i}(G/B) \longrightarrow H^{i}(\mathbb{P}^{2})
\]
be the map induced at the level of cohomology. 
We have from \ref{borelisom}
\[ H^{\bullet}(G/B) \cong \mathbb{Z}[x_1,x_2,\ldots,x_n] /{\cal{I}} 
\]
where ${\cal{I}}$ is the proper ideal of $\mathbb{Z}[x_1,x_2,\ldots,x_n]$ consisting of elementary symmetric polynomials. We have $x_i$ lies in $H^2(G/B)$. In other words, degree of $x_i$ is $2$. And we have, 
\[ H^{\bullet}(\mathbb{P}^{2}) \cong \mathbb{Z}[t]/t^3.
\] 
where degree of $t$ is $2$. Since $\phi^*(H^{2}(G/B)) \subseteq H^{2}(\mathbb{P}^2)$, we can assume
\[ \phi^*(x_i) = a_it
\] for some $a_i \in \mathbb{Z}$. Since ${\cal{I}}$ is generated by power sum symmetric polynomials, we have 
\[ \sum x_i^2 = 0 
\] 
in $H^{\bullet}(G/B)$. Thus in the image we will have,
 \[\sum a_i^2 = 0.
 \] 
Since $a_i$ are all integer we have $a_i = 0$ for all $i$. Therefore $\phi^{*i} = 0$ for all $i >0$. Hence, the map $\phi$ is a constant map. 
\end{proof}

\begin{corollary} \label{h/b_h} Let $H$ be a reductive group and $B_H$ be a Borel subgroup of $H$. Then there is no non constant morphism from $\mathbb{P}^2$ to $H/B_H$.
\end{corollary}
\begin{proof} Choose a faithful representation of $H$ in $SL(m,\mathbb{C})$ such that $B_H$ maps to a Borel subgroup $B$ of $SL(m,\mathbb{C})$. So we get a embedding of $H/B_H$ inside $SL(m,\mathbb{C})/B$. We now use theorem \ref{consp2} to conclude the proof.
\end{proof}

\begin{corollary}\label{gratogb} A morphism from $Gr(r,s)$ where $s \geq 3$ to $G/B$ is constant. 
\end{corollary}
\begin{proof} Since $\rm{Pic}(Gr(r,s))$ is $\mathbb{Z}$, we have every map from $Gr(r,s)$ to a projective variety is either finite or constant. Since $\mathbb{P}^2$ sits inside $Gr(r,s)$ whenever $s \geq 3$ and we have only constant morphism from $\mathbb{P}^2$ to $G/B$, the maps from $Gr(r,s)$ to $G/B$ must be constant as well.
\end{proof}
Let $V$ be a vector space of dimension $n$ and 
\[ 1 \leq i_1 < i_2 < \cdots < i_k = n. 
\] 
be a sequence of integers. We define $G(i_1,i_2,\ldots,i_k)$ the partial flag variety $G/P$ consisting of linear subspaces $L_{i_1}, L_{i_2},\ldots, L_{i_k}$ of $V$ such that $L_{i_j} \subset L_{i_{j+1}}$ and $\mathrm{dim} (L_{i_j}) = i_j$. 

\begin{remark} If $k = 2$ and $i_1 = d$ we obtain $G(i_1,i_2)$ as the Grassmannian variety $Gr(d,n)$. The full flag variety $G/B$ is obtained by choosing $i_j = j$. And any partial flag variety $G/P$ where $P$ contains $B$ can be obtained this way. 
\end{remark}

\begin{lemma}\label{p2togp}  There exists a $Gr(r,s)$ with $s \geq 3$ passing through each point of $G/P$ where $P$ is a parabolic subgroup which is not a Borel subgroup. 
\end{lemma}
\begin{proof} 
Since $P$ is not a Borel subgroup we have $n \geq 3$.
We are already done for the case of Grassmannian variety $Gr(d,n)$. So we can assume $k >2$ and $G/P = G(i_1,i_2,\ldots,i_k)$.
If $P$ is not $B$ then there either $i_1 > 1$ or $i_1 = 1$ and there exists a smallest $j$ such that $i_{j+1} > i_{j} + 1$. If $i_1 > 1$, then we have the fibers of the projection 
\[ G(i_1,i_2,\ldots,i_k) \longrightarrow G(i_2,i_3,\ldots,i_k)
\]
is $Gr(i_1,i_2)$ with $i_2 \geq 3$, hence we are done.

If $i_1 = 1$, choose the smallest $j$ such that $i_j = j$ and $i_{j+1} >  j +1$. If $j = 1$, ie. $i_2 > 2$, we have the fibres of the projection
\begin{equation}\label{eqngi}
 G(1,i_2,\ldots, i_k) \longrightarrow G(i_2,i_3,\ldots, i_k)
\end{equation}   
is  $\mathbb{P}^{r-1}$ where $r = i_2-1 \geq 2$. If $j \geq 2$, then we have the fibre of 
\begin{equation}\label{eqngi}
 G(i_1,i_2,\ldots, i_k) \longrightarrow G(i_1,i_2,\ldots,i_{j-1},i_{j+1},\ldots, i_k)
\end{equation} 
is $\mathbb{P}^{r-1}$ where $ r = i_{j+1}-i_{j-1} \geq 3$.

 This proves the lemma. 
\end{proof}

\begin{remark} \label{pg2} Let $P$ be a parabolic subgroup which is not a maximal parabolic or a Borel. The proof of the above lemma provides a $Gr(r,s)$-fibration $G/P \rightarrow G/P'$ for some $s \geq 3$ where $P'$ is a parabolic subgroup containing $P$.
\end{remark}
\begin{corollary} \label{gmptohb} Let $H$ be a reductive group and $B_{H}$ be a Borel subgroup of $H$. Fix a parabolic subgroup $P$ of $G$ and a non constant morphism $\phi: G/P \rightarrow H/B_{H}$. Then $P$ is a Borel subgroup. 
\end{corollary}

\begin{proof} We know that any $H/B_H$ embeds inside a $SL(N,\C)/B$ for some $N$, where $B$ is a Borel subgroup of $SL(N,\C)$. So we are reduced to the case where $H =SL(N,\C)$ and $B_H$ is a Borel subgroup of $SL(N,\C)$.

We assume on the contrary that $P$ is not a Borel subgroup. If $P$ is a maximal parabolic subgroup then by lemma \ref{gratogb} the map $\phi$ must be constant.

We can therefore assume $P$ not a Borel or a maximal parabolic subgroup. From corollary \ref{pg2} we obtain a parabolic subgroup $P'$ containing $P$ such that $G/P \rightarrow G/P'$ is $Gr(r,s)$-fibration with $s \geq 3$. Since $\phi$ is constant on $Gr(r,s)$ we have $\phi$ factors through $G/P'$. Repeating the argument we can assume that $\phi$ factors through a $G/Q$ where $Q$ is a maximal parabolic subgroup and hence we conclude that $\phi$ is constant.

\end{proof}

\begin{remark}  Shrawan Kumar \cite{kumarsconj} has extended corollary \ref{gmptohb} to an arbitrary simple group $G$. 
\end{remark}

\section{Maps from $\mathbb{P}^3$ to $G/P$ for a minimal parabolic subgroup}\label{sec:P3}

We assume the notations from the previous sections. We thus have $P_{\alpha}$ the minimal parabolic subgroup $B \cup Bs_{\alpha}B$. When $\alpha = \alpha_1$ we will show that there is no non constant morphism from $\mathbb{P}^3$ to $G/{P_{\alpha}}$. Since $G/P_{\alpha_1} \cong G/P_{\alpha_{n-1}}$ we conclude that there is no non constant morphism from $\mathbb{P}^3$ to $G/{P_{\alpha_{n-1}}}$ as well.  However, when we have any other minimal parabolic subgroup $P_{\alpha_{j}}$, $j\neq 1,n-1$ we will show that there are non constant morphisms from $\mathbb{P}^3$ to $G/{P_{\alpha_j}}$.

Fix a basis $e_1, e_2, \ldots, e_n$ of $V$. Define the subspaces $M_i$ to be the span of $e_1, e_2 \ldots, e_i$. Let $D_k$ denote the Schubert divisor in the $Gr(k,n)$ which is defined as 
\[ D_k = \{ F \in Gr(k,n) | F \cap M_{n-k} \neq 0 \}. 
\] 
We define the following two codimension 2 Schubert subvarieties of the Grassmannian $Gr(k,n)$:  
\[ D_{k,k+1} := \{ F \in Gr(k,n) | ~F \cap M_{n-k-1} \neq 0 \} 
\] 
\[ D_{k,k-1} := \{ F \in Gr(k,n) | ~\rm{dim}(F \cap M_{n-k+1} \geq 2) \}.
\] 
We note that $D_{n,n+1}$ and $D_{1,0}$ are empty sets.
We prove the following lemmas.

\begin{lemma}\label{lm41} Let $1\leq k \leq n$. We have the following relation in $H^4(Gr(k,n))$ 
\[D_{k}.D_{k} = D_{k,k-1} + D_{k,k+1}.\]
\end{lemma}
\begin{proof} To prove the lemma we would intersect the Schubert divisors fixing two different complimentary $n-k$ dimensional vector subspaces. Let $M'_{n-k}$ is the vector space generated by $M_{n-k-1}$ and $e_{n-k+1}$. Let $D'_k := \{ F \in Gr(k,n) | ~F \cap M'_{n-k} \neq 0 \} $ be the divisor linearly equivalent to $D_k$ defined with respect to $M'_{n-k}$. Then we can see that 
\[ D_k \cap D'_k = \{ F\in Gr(k,n) | ~F \cap M_{n-k} \neq 0\} \cap \{F \in Gr(k,n)|~ F\cap M'_{n-k} \neq 0\}
\]
\[ =\{F\in Gr(k,n) | F\cap M_{n-k-1} \neq 0 \} \cup \{\rm{dim}(F \cap M_{n-k+1}) \geq 2\}
\]
which by definition is $D_{k,k+1} \cup D_{k,k-1}$. Hence, the lemma follows.
\end{proof}

\begin{lemma}\label{lm42} We have the following relation in the cohomology $H^4(G(k,k+1,n))$ 
\[D_{k}.D_{k+1} = D_{k,k+1} + D_{k+1,k}.\]
\end{lemma}
\begin{proof} We note that the intersection of $D_{k}$ with $D_{k+1}$ is happening at $G(k,k+1,n)$. $D_{k+1}$ is linearly equivalent to $\{ (F,E) \in G(k,k+1,n) | ~ E \cap M_{n-k-1} \neq 0\}$ in $G(k,k+1,n)$. We observe that both $D_{k,k+1}$ and $D_{k+1,k}$ lie in the intersection of $D_{k}$ and $D_{k+1}$. If we choose a $F$ from the intersection not in $D_{k,k+1}$ we observe that $F \cap M_{n-k-1} = 0$ and $F\cap M_{n-k} \neq 0$. Then $F \cap M_{n-k}$ and $E\cap M_{n-k-1}$ are non zero and linearly independent, so they span atleast two dimensional vector space and it is contained in $E \cap M_{n-k}$. So we have, 
 \[ \{ F \in Gr(k,n) |~F\cap M_{n-k} \neq 0 \} \cap \{ (F,E) \in G(k,k+1,n) | ~ E \cap M_{n-k-1} \neq 0\}\]
\[ = \{F \in Gr(k,n) |~F \cap M_{n-k-1} \neq 0 \} \cup \{ E\in Gr(k+1,n) |~\rm{dim} (E \cap M_{n-k}) \geq 2\}.  
\] 
Hence, the lemma follows.
\end{proof}

Let $E_1$ denote the codimension $3$ Schubert cycle defined by the Schubert variety  $\{F\in Gr(2,n) |~F \cap M_{n-3} \neq 0 ~\rm{and} ~F \subset M_{n-1}\}$. Let $E_2$ be the codimension $3$ Schubert cycle defined by the Schubert variety $\{F\in Gr(2,n) |~F \cap M_{n-4} \neq 0 \}$. 

\begin{lemma}\label{lm3} We have the following relations in $H^6(Gr(2,n))$ : 
\begin{itemize}
\item[(i)] $D_{2,1}. D_2 = E_1$.
\item[(ii)] $D_{2,3}. D_2 = E_1 + E_2$. 
\end{itemize}
\end{lemma}
\begin{proof} (i) Let $M_{n-2}''$ be the $n-2$ dimensional vector space spanned by $e_1,e_2, \ldots, e_{n-3}, e_{n}$. Let $D_2''$ be the divisor linearly equivalent to $D_2$ defined by 
$\{ F \in Gr(2,n) | ~F \cap M''_{n-2} \neq 0 \}$. We have $F \subset M_{n-1}$ as it is in $D_{2,1}$. It follows that $F \cap M_{n-2}'' \subset M_{n-1} \cap M_{n-2}'' = M_{n-3}$ is nonzero. Therefore, $F \cap M_{n-3} \neq 0$.

(ii) Let $M_{n-2}'''$ be the $n-2$ dimensional vector space spanned by $e_1,e_2,\ldots, e_{n-4},e_{n-2},e_{n-1}$. Let $D_2'''$ be the divisor linearly equivalent to $D_2$ defined by 
$\{ F \in Gr(2,n) | ~F \cap M'''_{n-2} \neq 0 \}$. Let $F$ be in the intersection of $D_{2,3}$ and $D'''_{2}$. If $F \cap M_{n-4} \neq 0 $ then $F$ is the component $E_2$. So we assume $F \cap M_{n-4} = 0$. Notice that $M_{n-4} = M_{n-2} \cap M_{n-2}'''$. But on the other hand $F\cap M_{n-2} \neq 0 $ and $F\cap M'''_{n-2} \neq 0 $, therefore $F$ is contained in the span of $M_{n-2}$ and $M_{n-2}'''$ which is $M_{n-1}$. Hence $F$ is contained in $M_{n-1}$, i.e $F \in E_1$. Hence the lemma.

\end{proof}

Since the map $H^{*}(Gr(k,n))$ to $H^{*}(G/B)$ is injective the above relations holds in $H^*(G/B)$ as well. We use the same notations $D_i$ and $D_{i,j}$ to define the Schubert classes in $H^*(G/B)$. Note that the above relations can also be deduced from Monk's formula. 

\begin{theorem}\label{consp3}
There is no nonconstant morphism from $\mathbb{P}^3$ to $G/{P_{\alpha_1}}$. 
\end{theorem}

\begin{proof}
Let $P = P_{\alpha_1}$. So $G/P$ is $G(2,3,\ldots,n)$. Let
\[
\phi : \mathbb{P}^3 \longrightarrow G/P 
\] be a  map. Let  
\[ \phi^{*}: H^{\bullet}(G/P) \longrightarrow H^{\bullet}(\mathbb{P}^{3})
 \] 
 be the map at the level of cohomology. Let
 \[ \phi^{*i}: H^{i}(G/P) \longrightarrow H^{i}(\mathbb{P}^{3})
 \] be the map at degree $i$. 
 We know that $H^{*}(\mathbb{P}^{3}) \cong \mathbb{Z}[t]/(t^4)$.
 We will show that $\phi^{*i} = 0$ for all $i > 0$.

We know that the divisors in $G(2,3,\ldots,n)$ are $D_2, D_3, \ldots, D_{n-1}$. From \ref{lm41} and \ref{lm42} we have the following relations in $H^{*}(G/P)$
\begin{align*}
D_2D_2 &= D_{2,1} + D_{2,3} \\
D_2D_3 &= D_{2,3} + D_{3,2} \\
D_3D_3 &= D_{3,2} + D_{3,4} \\
&\;\;\vdots \notag \\
D_{n-2}D_{n-1} &= D_{n-2,n-1} + D_{n-1,n-2}\\
D_{n-1}D_{n-1} &= D_{n-1,n-2} \\
\end{align*}

Letting $\phi^{*}(D_i) = a_{i}t$ in $H^2(\mathbb{P}^3)$ and  $\phi^{*}(D_{i,j}) = b_{i,j}t^2$ in $H^4(\mathbb{P}^3)$ we obtain the following relations in $H^{*}(\mathbb{P}^3)$. 
\begin{flalign*}
a_2^2 &= b_{2,1} + b_{2,3}\\
a_2a_3 &= b_{2,3} + b_{3,2} \\
 &\;\;\vdots \notag \\
a_{n-1}^2 &= b_{n-1,n-2}
\end{flalign*}

So rewriting $b_{i,j}$ in terms of $a_{i,j}$ we obtain 
\begin{align}
b_{n-1,n-2} &= a_{n-1}^2 \nonumber \\
b_{n-2,n-1} &= a_{n-1}a_{n-2} - a_{n-1}^2 \nonumber \\
&\;\;\vdots \notag \\
b_{i,i-1} &= (a_{i}^2 + a_{i+1}^2+ \ldots a_{n-1}^2) - (a_{i}a_{i+1} + a_{i+1}a_{i+2} \ldots + a_{n-2}a_{n-1}) \nonumber  \\
b_{i-1,i} &= (a_{i-1}a_{i} + a_{i}a_{i+1} + \ldots + a_{n-2}a_{n-1}) -(a_{i}^2 + a_{i+1}^2+ \ldots a_{n-1}^2) \nonumber \\
 &\;\;\vdots \notag \\
b_{2,3} &= (a_2a_3 + a_3a_4 + \ldots + a_{n-2}a_{n-1}) - (a_{3}^2 + \ldots + a_{n-1}^2) \nonumber  \\ 
b_{2,1} &= (a_{2}^2 + \ldots + a_{n-1}^2) - (a_2a_3 + a_3a_4 + \ldots + a_{n-2}a_{n-1}). \nonumber \\
\nonumber 
\end{align}

Let $\phi^{*}(E_1) = c_1t^3$ and $\phi^{*}(E_2) = c_2t^3$. Therefore from \ref{lm3} we get 
\[ b_{2,1} a_2 = c_1 \] 
\[ b_{2,3 }a_2 = c_1 + c_2 .\]

We know Schubert classes are represented by algebraic cycles and hence their pullbacks are algebraic cycles in the projective space. Therefore, Schubert polynomials are mapped to non negative classes in the cohomology of projective spaces. So $c_2 \geq 0$. Therefore, we obtain $b_{2,3 }a_2 \geq b_{2,1}a_2$. We have $a_2 \geq 0$. If $a_2 >0 $ we observe $b_{2,3 } \geq b_{2,1}$. Hence, 
\[ (a_2a_3 + a_3a_4 + \ldots + a_{n-2}a_{n-1}) - (a_{3}^2 + \ldots + a_{n-1}^2) \geq (a_{2}^2 + \ldots + a_{n-1}^2) - (a_2a_3 + a_3a_4 + \ldots + a_{n-2}a_{n-1})  
\] which implies that 
\[ (a_2- a_3)^2 + (a_3-a_4)^2 + \ldots (a_{n-2} - a_{n-1})^2 + a_{2}^2 \leq 0
\] which forces $a_i = 0$ for all $i$.

If $a_2 = 0$ we have $b_{2,3} = b_{2,1} = 0$. Then we obtain 
\[ a_{3}^2 + a_{4}^2 + \ldots + a_{n-1}^2  = a_{3}a_{4} + \ldots a_{n-2}a_{n-1}. 
\]
which implies 
\[ (a_3-a_4)^2 + (a_4 - a_5)^2 + \ldots + (a_{n-2} - a_{n-1})^2 + a_{3}^2 + a_{n-1}^2 =  0
\] which forces all $a_i = 0$.
Therefore, we conclude that $\phi^{*i} =0$ for all $i > 0$.
\end{proof}

\begin{corollary} There is no nonconstant morphism from $\mathbb{P}^3$ to $G/{P_{\alpha_{n-1}}}$. 
\end{corollary}

\begin{proof} 

Since $G= SL(n,\mathbb{C})$, we have an automorphism of $G$ which is induced by the Dynkin involution taking $\alpha_i$ to $\alpha_{n-i}$ for all $ 1\leq i \leq n-1$. Under this automorphism we have $P_{\alpha_1}$ isomorphic to $P_{\alpha_{n-1}}$. We have $G/P_{\alpha_1}$ isomorphic to $G/P_{\alpha_{n-1}}$.

\end{proof}

\begin{lemma} \label{p3tog134}
 There is a non constant morphism from $\mathbb{P}^3$ to $G(1,3,4)$. 
\end{lemma}

\begin{proof}
Let $V$ be a vector space of dimension $4$. Let $\mathbb{P}^3$ be the projective space of lines in $V$. Fix a non-degenerate alternating bilinear form. Because the form is non-degenerate and alternating it follows that for every line $L$ the orthogonal compliment $L^{\perp}$ of $L$ is a $3$ dimensional subspace of $V$ containing $L$. Hence, $(L,L^{\perp}, V)$ is an element of $G(1,3,4)$ and the map $L \mapsto (L,L^{\perp}, V)$ defines the required morphism.
\end{proof}

\begin{theorem} There is a non constant morphism from $\mathbb{P}^3$ to $G/P_{\alpha}$ for all minimal parabolic subgroup  $P_{\alpha}$ with  ${\alpha} \notin \{ {\alpha_{1}, \alpha_{n-1}} \}$.
\end{theorem}

\begin{proof} Let $\alpha = \alpha_j$ where $2 \leq j \leq n-2$. Fix a flag
 \[ L_1 \subset L_2 \cdots \subset L_{j-2} \subset L_{j+2} \subset L_{j+3} \cdots  \subset L_{n-1} \subset L_n 
 \] 
 where dimension of $L_j = j$.
Then the fiber over this flag of the map 
\[ G(1,2,\ldots,j-1,j+1,j+2,\ldots,n-1,n) \longrightarrow G(1,2,\ldots,j-2, j+2,\ldots, n-1, n )
\]
is isomorphic to $G(1,3,4)$ which is identified as the flags $(L_{j-1}/L_{j-2},L_{j+1}/L_{j-2},L_{j+2}/L_{j-2})$. So we have a map from $G(1,3,4)$ to $G/P_{\alpha}$. And using lemma \ref{p3tog134} we prove the theorem.

\end{proof}

\section{Maps from $\mathbb{P}^4$ to $G/P$ for a minimal parabolic subgroup}\label{sec:P4}

The first arXiv version of this paper contained an incorrect formula in the argument for $\Pj^4$.
We are grateful to Yanjie Li for pointing this out.
In this section we give a corrected proof of the nonexistence of morphisms from $\Pj^4$
to minimal parabolic quotients.
The argument below is completely elementary and relies only on the Borel presentation
of cohomology and basic inequalities among symmetric polynomials.

\begin{theorem}
There is no non-constant morphism from $\mathbb{P}^4$ to $G/P$, where $P$ is a minimal parabolic subgroup of $G$.
\end{theorem}

\begin{proof}
We have already seen that there is no map form $\mathbb{P}^3$ to $G/P_{\alpha_i}$ for $i = 1, n-1$. Without loss of generality, we take $P = P_{\alpha_2}$.  
We show that there is no non-zero homomorphism 
\[
H^*(G/P) \longrightarrow H^4(\mathbb{P}^4).
\]

Recall that the Borel presentation of the cohomology ring can be written as
\[
H^*(G/P) \;=\; \mathbb{Z}[x_1, x_2, y, \ldots, x_n] / \mathcal{I},
\]
where $\mathcal{I}$ is generated by the coefficients of the polynomial $P(t) - 1$, with
\[
P(t) = (1 + x_1 t)\,(1 + x_2 t + y t^2)\,(1 + x_3 t) \cdots (1 + x_n t).
\]

Let 
\[
\phi : \mathbb{P}^4 \longrightarrow G/P
\]
be a morphism, and let
\[
\phi^* : H^\bullet(G/P) \longrightarrow H^\bullet(\mathbb{P}^4)
\]
be the induced map on cohomology.  
Under $\phi^*$, we set
\[
x_i \longmapsto a_i, \qquad y \longmapsto b.
\]

Denote by $e_i$ the elementary symmetric polynomials in the $a_j$:
\[
e_1 = \sum_i a_i, \quad 
e_2 = \sum_{i<j} a_i a_j, \quad 
e_3 = \sum_{i<j<k} a_i a_j a_k, \quad 
e_4 = \sum_{i<j<k<\ell} a_i a_j a_k a_\ell.
\]
Let $\hat e_i$ be the corresponding symmetric polynomial omitting $a_2$.  
The relations obtained from the defining equation $P(t) \equiv 1 \pmod{t^5}$ are then
\[
e_1 = 0, \qquad 
e_2 + b = 0, \qquad 
e_3 + b \hat e_1 = 0, \qquad 
e_4 + b \hat e_2 = 0.
\]

Hence
\[
\hat e_1 = -a_2, \qquad 
e_2 = -b, \qquad 
e_3 = a_2 b, \qquad 
e_4 = b^2 - b a_2^2.
\]

Let $p_k = \sum_i a_i^k$ denote the power-sum symmetric polynomials.  
By Newton’s identities, we have:
\[
\begin{aligned}
p_1 &= e_1 = 0,\\
p_2 &= e_1 p_1 - 2 e_2 = 2b,\\
p_3 &= e_1 p_2 - e_2 p_1 + 3 e_3 = 3 a_2 b,\\
p_4 &= e_1^4 - 4 e_1^2 + 4 e_1 e_3 + 2 e_2^2 - 4 e_4 
     = 4 b a_2^2 - 2 b^2.
\end{aligned}
\]

By the Cauchy–Schwarz inequality,
\[
p_3^2 \leq p_2 p_4.
\]
Substituting the above expressions gives
\[
(3 a_2 b)^2 \leq (2b)(4 b a_2^2 - 2 b^2),
\]
which simplifies to
\[
9 a_2^2 b^2 \leq 8 a_2^2 b^2 - 4 b^3
\quad \Longrightarrow \quad
b^2(a_2^2 + 4b) \leq 0.
\]
Thus $a_2^2 + 4b \leq 0$, implying $b \leq 0$.  

On the other hand, since $p_2 = 2b = \sum_i a_i^2 \ge 0$, we must have $b \ge 0$.  
Hence $b = 0$, which forces $p_2 = 0$, and consequently $a_i = 0$ for all $i$.  

Therefore, $\phi^*$ is the zero map, and $\phi$ is constant.
\end{proof}
\eat{
We already observed that $H/B_H$ sits inside $G/B$ for any reductive algebraic group $H$ with Borel subgroup $B_H$. Noting from the fact that $H/B_H$ is a $\mathbb{P}^1$ bundle over $H/P_H$ for any minimal parabolic subgroup $P_H$ containing $B_H$, we conclude $H/P_H$ should sit inside $G/P$ for a minimal parabolic subgroup $P$ of $G$. So we can conclude 
\begin{corollary} Let $H$ be a reductive group with $P_H$ minimal parabolic subgroup. Then, there is no morphism from $\mathbb{P}^4$ to $H/P_H$.
\end{corollary}
}
\section{Declaration}
\textbf{Conflicts of interest} 
The authors declare that they have no conflicts of interest.

\bibliographystyle{amsplain}
\bibliography{references}

\end{document}